\documentclass[12pt,reqno,a4paper]{amsart}
\usepackage{mathrsfs, tikz}
\usepackage{amssymb,epsfig,graphics}
\usepackage{amsmath}
\usepackage{bbold}
\usepackage{a4wide}
\usepackage{hyperref}
\usepackage[normalem]{ulem}
\usepackage{tikz}
\newcommand{\Hquad}{\hspace{0.5em}}

\newtheorem{theorem}{Theorem}[section]
\newtheorem{proposition}[theorem]{Proposition}
\newtheorem{definition}[theorem]{Definition}
\newtheorem{corollary}[theorem]{Corollary}
\newtheorem{lemma}[theorem]{Lemma}
\newtheorem{sub-lemma}[theorem]{Sub-Lemma}

\newtheorem{remark}[theorem]{Remark}

\def\A{\mathcal{A}}

\def\C{\mathcal{C}}

\def\R{\mathcal{R}}
\def\P{\mathcal{P}}

\def\Q{\mathcal{Q}}
\def\D{\Delta}
\def\NN{\mathbb{N}}
\def\PP{\mathbb{P}}

\def\ZZ{\mathbb{Z}}

\DeclareMathOperator{\diam}{diam}

\let\eps=\varepsilon

\def\E{\mathcal{E}}
\def\D{\mathcal{D}}

\def\NN{{\mathbb N}}

\newcommand{\Dom}{\Delta_\omega}

\begin{document}
\title{Quenched decay of correlations for one dimensional random Lorenz maps}
\author{Andrew Larkin}
\address{Department of Mathematical Sciences, Loughborough University,
Loughborough, Leicestershire, LE11 3TU, UK}
\email{a.larkin2@lboro.ac.uk}
\thanks{}
\begin{abstract}
We study rates of mixing for small random perturbations of one dimensional Lorenz maps. Using a random tower construction, we prove that, for H\"older observables, the random system admits exponential rates of quenched correlation decay.
\end{abstract}
\date{\today}
\maketitle
\section{Introduction}
In 1963, Lorenz \cite{Lo} introduced the following system of equations 
\begin{equation}\label{LS}
\dot x = -10x+10y, \,\, \dot y = 28x-y-xz, \,\, \dot z = -\frac{8}{3}z+xy
\end{equation}
as a simplified model for atmospheric convection. Numerical simulations performed by Lorenz showed that the above system exhibits sensitive dependence on initial conditions and has a non-periodic ``strange" attractor.  Since then, \eqref{LS} became a basic example of a chaotic deterministic system that is notoriously difficult to analyse. We refer to \cite{AP} for a thorough account on this topic.

It is now well known that  a useful technique to analyse the statistical properties of such a flow, and any nearby flow in the $C^2$-topology, is to study the dynamics of the flow restricted to a Poincar\'e section via a well defined Poincar\'e map \cite{APPV}. Such a Poincar\'e map admits an invariant stable foliation; moreover, it is strictly uniformly contracting along stable leaves \cite{APPV}. Therefore, the dynamics of the Poincar\'e map can be understood via quotienting along stable leaves; i.e., by studying the dynamics of its one dimensional quotient map along stable leaves. The above technique have been employed to obtain statistical properties of Lorenz flows \cite{HM} and to prove that  such statistical properties of this family of flows is stable under deterministic perturbations \cite{AS, BR, BMR}. This illustrates the importance of understanding the statistical properties of one-dimensional Lorenz maps. 

One-dimensional Lorenz maps have been thoroughly studied in the literature. In \cite{DO}, which is the main inspiration for this paper, a deterministic Lorenz-like map is studied for which it is proved there is exponential decay of correlation. Additionally, in \cite{BR} they examine a perturbed family of Lorenz-like maps with differing singularity points and show statistical stability. Another example is \cite{AS1}, where they examine a family of perturbed contracting Lorenz-like maps, in contrast to the expanding case. They show that the set of points that have not achieved exponential growth in the derivative or slow recurrence to the critical point at a given time decays exponentially, which thereby implies further statistical properties.
 
In this paper, we study rates of mixing for small random perturbations of such one dimensional maps. We use a random tower construction to prove that for H\"older observables the random system admits exponential rates of quenched correlation decay. 
The paper is organised as follows. In Section \ref{sec2} we present the setup and introduce the random system under consideration. Section \ref{sec2} also includes the main result of the paper, Theorem \ref{thm:1d}. Section \ref{proofs} includes the proof of Theorem \ref{thm:1d}. Section \ref{appendix} is an Appendix, which contains a version of the abstract random tower result of \cite{BBMD, Du}, which is used to deduce Theorem \ref{thm:1d}.

\noindent{\bf Acknowledgment}: I would like to thank Marks Ruziboev for our helpful discussions throughout this work. I would also like to thank anonymous referees whose comments improved the presentation of the paper.

\section{Setup}\label{sec2}

\subsection{Our unperturbed system} \label{sec:unperturbed}  
We assume that the following conditions hold.
\begin{itemize}
\item[(A1)] $T_0:I\to I$, $I=[-\frac12, \frac12]$, is $C^{1}$  on $I\setminus\{0\}$  with a singularity at $0$ and one-sided limits 
\mbox{$T_0(0^+)<0$} and $T_0(0^-)>0$. Furthermore, $T_0$ is uniformly expanding on $I\setminus\{0\}$, i.e.
		there are constants $\tilde C>0$ and $\ell>0$  such that $DT_0^n(x)> \tilde{C} e^{n \ell}$ for all $n\ge1$ whenever $x\notin \bigcup_{j=0}^{n-1}(T_0^{j})^{-1}(0)$;
\item[(A2)] There exists $C>1$ and $0<\lambda_0<\frac12$ such that in a one sided neighborhood of $0$ 
\begin{align}C^{-1}|x|^{\lambda_0-1}\le |DT_0(x)|\le C|x|^{\lambda_0-1}. \label{eq1}\end{align}
Moreover, for any $\alpha \in (0,1)$, $1/DT_0$ is $\alpha$-H\"older on $[-\frac12, 0)$ and $(0,  \frac12]$ with $\alpha$-H\"older constant $K=K(\alpha)$;
\item[(A3)] $T_0$ is transitive (for the construction, we use that pre-images of $0$ is dense in $I$).\\
\end{itemize} 

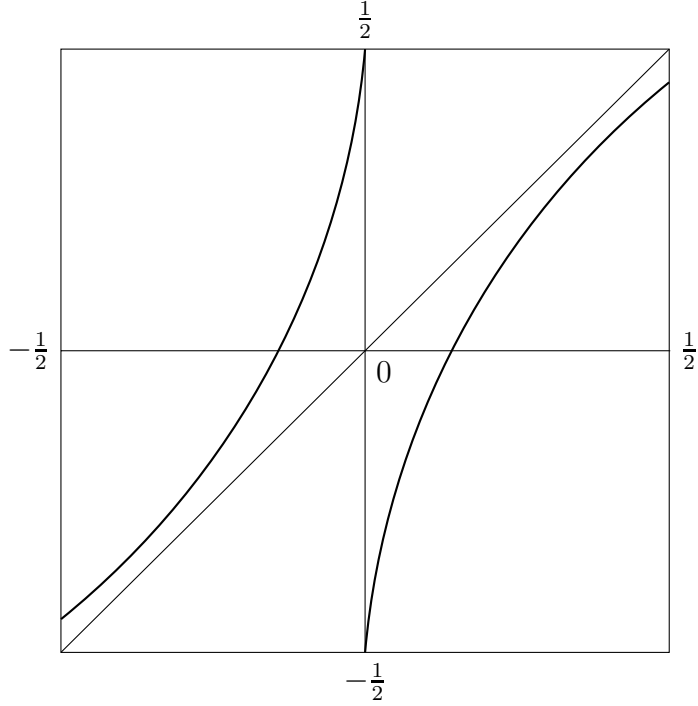
\begin{figure}
\begin{center}
\begin{tikzpicture}
	\draw (0,0) rectangle (8,8);
	\draw (4,0) -- (4, 8) ;
	\draw (0,4) -- (8, 4)  ;
	\draw (0,0) -- (8, 8);
	\draw[thick] (4,8) arc (-5:-50.75:11cm);
	\draw[thick] (4,0) arc (175:129.25:11cm);
  	\draw (0,4) -- (0 ,4) node[anchor=east] {$-\frac{1}{2}$};
	\draw (4,4) -- (4,4) node[anchor=north west] {$0$};
	\draw (8,4) -- (8,4) node[anchor=west] {$\frac{1}{2}$};
	\draw (4,8)  node[anchor=south] {$\frac{1}{2}$};
	\draw (4,0)  node[anchor=north] {$-\frac{1}{2}$};
\end{tikzpicture}
\end{center}
\caption{The one-dimensional Lorenz-like map $T_0$.} \label{fig:1}
\end{figure}
See Figure \ref{fig:1} on the next page for a visual depiction of this graph. Notice that (A2) implies that for all $x,y \in I$ we have
\begin{equation}\label{eq:2a}
|DT_0(x) - DT_0(y)| \le K' \frac{|x-y|^{\alpha}}{|x|^{1 - \lambda_0}|y|^{1 - \lambda_0}},
\end{equation}
where $K' = C^2 K$. Furthermore, if $1 - \lambda_0 \le \alpha$, then $DT_0$ is \textit{locally H\"older}, in the sense that for all $x,y \in I$ we have  
\begin{equation}\label{eq2b}
|DT_0(x) - DT_0(y)| \le K' \frac{|x-y|^{\alpha}}{|x|^{\alpha}|y|^{\alpha}}. 
\end{equation}

\begin{definition}\label{def:cont}
Let $\{T_\varepsilon : I \to I  \}_{\varepsilon > 0}$ be a family of maps. For some fixed $\alpha \in (0,1)$, we say $T_\eps$ is $C^{1 + \alpha}$ on $I \setminus \{0 \}$  if $T_\eps$ is $C^1$ on $I \setminus \{0 \}$ and $\frac{1}{DT_\eps}$ is $\alpha$-H\"older on $I \setminus \{0 \}$. Furthermore, we say $\{T_\varepsilon : I \to I  \}_{\varepsilon > 0}$  is a continuous family of $C^{1 + \alpha}(I \setminus \{0 \})$ maps if 
 $$\lim_{\varepsilon \to \varepsilon_0} \| T_\varepsilon - T_{\varepsilon_0} \|_{\infty} = 0,$$
 and
 $$ \lim_{\varepsilon \to \varepsilon_0}\Big\|\frac{1}{DT_\varepsilon}-\frac{1}{DT_{\varepsilon_0}}\Big\|_{C^\alpha}=0.$$
\end{definition}

\subsection{Our perturbed system}\label{ssec:1drd} Let  $T_0:I\to I$  be the  map introduced in subsection \ref{sec:unperturbed}. Notice that $T_0$ is $C^{1+\alpha}$ on $I \setminus \{0\}$ for any $\alpha \in (0,1)$ according to the definition in the previous subsection. Fix some $\alpha \in [1- \lambda_0,1)$. Define $\mathcal A(\eps)$ to be a $C^{1+\alpha}$ continuous family of maps containing $T_0$, in the sense of Definition \ref{def:cont}. Thus, there exists a sufficiently small interval  $[\lambda_0, \bar\lambda] \subset (0, \frac12)$, $\lambda_0<\bar\lambda$ with $T_\lambda\in\mathcal A(\eps)$ for any $\lambda\in [\lambda_0, \bar\lambda]$. Note that for $\lambda \in (\lambda_0, \bar\lambda]$, the map $T_\lambda$ satisfies the order of singularity condition (A2) with respect to $\lambda$ instead of $\lambda_0$.
Additionally, note that for any $\lambda\in [\lambda_0, \bar\lambda]$ we have
\begin{equation}\label{eq:2}
|DT_\lambda(x) - DT_\lambda(y)| < K' \frac{|x-y|^{\alpha}}{|x|^{\alpha}|y|^{\alpha}}.
\end{equation}



Let $P$ be the normalised Lebesgue measure on $[\lambda_0, \bar\lambda]$.  Let $\Omega=[\lambda_0, \bar\lambda]^{\ZZ}$, $\PP=P^{\ZZ}$ and $\sigma:\Omega\to\Omega$ be the left shift map; i.e., $\omega'=\sigma(\omega)$ if and only if $\omega'_{i}=\omega_{i+1}$ for all $i\in \ZZ$. Then $\sigma$ is an invertible map that preserves $\PP$.  Notice that $\omega$ denotes a bi-infinite sequence of parameter values from $[\lambda_0, \bar\lambda]$, i.e.
$$
\omega = \dots \omega_{-2} \omega_{-1} \omega_{0} \omega_{1} \omega_{2} \dots, \quad \omega_i \in [\lambda_0, \bar\lambda] \Hquad \forall i \in \mathbb Z.
$$\\

We express the random dynamics of our system in terms of the skew product 
$$S :\Omega \times I\to \Omega \times I,$$ where $$S(x,\omega)= (\sigma(\omega), T_{\omega_0}(x)).$$
Iterates of $S$ are defined naturally as
\[
S^n(x,\omega)= (\sigma^n(\omega), T^n_{\omega}(x)), \quad
T^n_{\omega}(x):=T_{\omega_{n-1}}\circ\dots\circ T_{\omega_0}(x).
\]
To simplify the notation, we denote $T_{\omega} = T_{\omega_0}$. We assume uniform expansion on random orbits: for every $\omega\in\Omega$,
\begin{equation}\label{eq:unifexp}
|DT^n_\omega(x)| > \tilde{C} e^{n \ell} \text{ for all } n\ge1,  \text{ for some }\tilde C>0, \ell>0,
\end{equation}
whenever $x\notin \bigcup_{j=0}^{n-1}(T_\omega^j)^{-1}(0)$.

Here we are looking at the quenched statistical properties of  $S$, i.e. we study statistical properties of the system generated by the  compositions $T^n_{\omega}$ on $I$ for almost every $\omega\in \Omega$, which we refer to as $\{T_\omega\}$ without confusion, since the underlying driving  process $(\Omega, \sigma, \PP)$ is fixed. 
We call a family of Borel probability measures $\{\mu_\omega\}_{\omega\in\Omega}$ on $I$ \emph{equivariant} if $\omega\mapsto\mu_\omega$ is measurable and 
\[
{T_\omega}_\ast\mu_\omega=\mu_{\sigma\omega} \text{ for }
\PP \text{ almost all }\omega\in\Omega.
\]\\
For fixed $\eta \in (0,1)$, let $\C^{\eta}(I)$ denote the set of $\eta$-H\"older functions on $I$, and let $L^\infty(I)$ denote the set of bounded functions on $I$. The following theorem is the main result of this paper.
\begin{theorem}\label{thm:1d} The random system $\{T_\omega\}$ admits a unique equivariant family of absolutely continuous probability measures $\{\mu_\omega\}$. Moreover, there exists a constant $ b >0$ such that for every $\varphi \in \C^{\eta}(I)$ and $\psi \in L^\infty(I)$, we have
\[
\left|\int(\varphi\circ T_\omega^n)\psi d\mu_\omega-\int \varphi d\mu_{\sigma^n\omega}\int \psi d\mu_\omega\right|\le C_{\varphi, \psi}  e^{-bn}
\] 
and 
\[
		\Big| \int (\varphi \circ T^n_{\sigma^{-n} \omega}) \psi d\mu_{\sigma^{-n}\omega} - \int \varphi d\mu_\omega \int \psi d\mu_{\sigma^{-n}\omega} \Big| \le C_{\varphi, \psi} e^{-bn}
\] 
for some constant $C_{\varphi, \psi}>0$ which only depends on $\varphi$ and $\psi$ and is uniform for all $\omega \in \Omega$. 

 \end{theorem} 
\section{Proof of theorem \ref{thm:1d}}\label{proofs} 
\subsection{Strategy of the proof}
We prove Theorem \ref{thm:1d} by showing that the random system $\{T_\omega\}$ admits a Random Young Tower structure  \cite{BBMD, Du} for every $\omega\in\Omega$ with exponential decay for the tail of the return times. Thus, uniform (in $\omega$) rates leads to uniform (in $\omega$) exponential decay of correlations. For convenience, we include the random tower theorem of  \cite{BBMD, Du} in the Appendix, which we will refer to in the proof of Theorem \ref{thm:1d}. Similar to the work of \cite{DO, DHL} in the deterministic  setting, to construct a random tower, we first construct an auxiliary stopping time called escaping time $E:J\to \NN$ on subintervals of $I$. Roughly speaking, $E$  is a moment of time when the image of a small interval $J'\subset J$ reaches a fixed interval length $\delta>0$ while at the same time $T_\omega^E(J')$ does not intersect a fixed neighborhood $\Delta_0$ of the singularity $0$. We show that $T^E_\omega$ has good distortion bounds and that $|\{E>n\}|$ decays exponentially fast as $n$ goes to infinity. The next step will be to construct a full return partition $\Q^\omega$ of some small neighborhood $\Delta^\ast\subset \Delta_0$ of the singularity. Once this is obtained, we induce a random Gibbs-Markov map $F_\omega$ in the sense of Definition 1.2.1 of \cite{Du}. This means we will need to show that  $F_\omega:  \Delta^*\to \Delta^*$, defined as $F_\omega(x) = T_\omega^{\tau_\omega(x)}(x)$, where $\tau_\omega:\Delta^*\to\mathbb N$ a measurable return time function, has the following properties: 
\begin{itemize}
\item $\tau_\omega|Q$ is constant for every $Q\in\Q^\omega$ and every $\omega\in\Omega$;
\item  $F_\omega|Q:Q \to \Delta^\ast $ is a uniformly expanding diffeomorphism with bounded distortion for every $Q\in\Q^\omega$ and every $\omega\in\Omega$; 
\item $ |\{\tau_\omega>n\}| \le Be^{-bn}$ for some constants $B>0$, $b\in(0, 1)$; 
\item  there exists $N_0\in \NN$ and two sequence $\{t_i\in\NN, i=1,2, \dots, N_0\}$, $\{\eps_i>0, i=1, \dots, N_0\}$ such that g.c.d.$\{t_i\}=1$ and $|\{x\in\Delta^\ast\mid \tau_\omega(x)=t_i\}|>\eps_i$ for almost every $\omega\in\Omega$;
\item on sets $\{ \tau_\omega = n \}$, $\tau_\omega$ only depends on the first $n$ elements of $\omega$, i.e. $\omega_0, \dots, \omega_{n-1}$.
\end{itemize}

\subsection{Escape time and partition}\label{ssec:Escape} Let us choose two sufficiently large constants $r_0,r_\ast>0$ satisfying $r_0<r_\ast$. Define $\delta^*=e^{-r_\ast}$ and  $\delta=e^{-r_0}$. Moreover, let  $\Delta_0=(-\delta, \delta)$ and $\Delta^\ast=(-\delta^*, \delta^*)$. Consider an exponential partition $\P=\{I_r\}_{r\in\ZZ}$ of $\Delta_0$ as in \cite{BC}, where   
\[
I_{r}=[e^{-(r+1)}, e^{-r}), \quad  I_{-r}=(-e^{-r}, -e^{-(r+1)}]\quad \text{ for }  r \ge r_0,
\]
and for $|r| < r_0$ we set $I_r=\emptyset$. Furthermore, fix $\vartheta=\left[\frac1\alpha\right]+1$, where $\alpha$ is the same as in \ref{ssec:1drd}.  We divide every $I_r$ into $r^\vartheta$ equal parts. Let $I_{r, m}$ denote one of the small intervals,  $m=1, \dots, r^\vartheta$.
We use the usual notation $x_{\omega,i}=T^i_\omega(x)$, $J_{ \omega, i}=T^i_\omega(J)$ for $i\ge 0$, $x \in I$ and any interval $J\subset I$. 

Let $J_0\subset  I$ be an interval such that either $J_0\cap \Delta_0=\emptyset$ and $|J_0|\ge \delta/5$, or $J_0=\Delta_0$. We wish to construct an \textit{escape partition} $\P^\omega(J_0)$ of $J_0$, and we also wish to construct a \textit{stopping time} function $E_\omega: J_0\to \NN$, which we call the \textit{escape time}, that has the following properties for every $\omega \in \Omega$:
\begin{itemize}
	\item $E_\omega|J$ is constant on each $J \in \P^\omega(J_0)$;
	\item $T_\omega^{E_\omega(J)}(J) \cap \Delta_0 =  \emptyset$ for each $J \in \P^\omega(J_0)$;
	\item $|T_\omega^{E_\omega(J)}(J)| \ge \delta$ for all $J \in \P^\omega(J_0)$;
	\item  for every $J\in \P^\omega(J_0)$ and for every time $j \le E_\omega(J)$ such that $T_\omega^j(J)\cap\Delta_0 \neq \emptyset$, $T_\omega^j(J)$ does not intersect more than three adjacent intervals of the form $I_{r, m}$, $|r|\ge r_0$, $m=1,\dots,  r^\vartheta$.
\end{itemize}

We have  the following proposition. 
\begin{proposition}\label{prop:Escape} 
Let $\delta>0$ be our previously chosen constant, and consider $J_0\subset I$ as described above. For every $\omega\in\Omega$ there exists a countable partition $\P^\omega(J_0)$ and an escape time $E_\omega:J_0\to \mathbb{N}$ that satisfies the properties mentioned above. Furthermore,  there exist constants $C_\delta>0$, $\gamma>0$ such that for any given time $n \in \mathbb{N}$ we have
\[
|\{E_\omega(J)\ge n: J\in\P^\omega(J_0) \}|\le C_\delta e^{-\gamma n}|J_0|.
\]
\end{proposition} 
We prove the proposition above in a series of lemmas. It is important to keep distortion under control, which means that we have to keep track of visits of the orbits near the singularity.  To this end, we use a chopping algorithm following \cite{BC}. 

Let us fix $\omega\in \Omega$ and $k\ge 1$, and suppose that the elements $J^* \in \P^\omega(J_0)$ satisfying $E_\omega(J^\ast)<n$ have already been constructed. Let $J \subset J_0$ be an interval in the complement of $\{E_\omega(J^*)<n\}$. If $T^k_\omega(J)\cap \Delta_0=\emptyset$, we call $k$ a \textit{free iterate} for $J$. Now, let us consider the following cases depending on the position and length of  $J_{\omega, k}=T^k_\omega(J)$, $k\in \NN$. \\

\noindent\textbf{Non-chopping intervals.} Suppose that $|J_{\omega, k}|<\delta$, $\Delta_0\cap J_{\omega, k}\neq \emptyset$ and $J_{\omega, k}$ does not intersect more than three adjacent intervals $I_{r, m}$.  Then $k$ is called an {\it inessential} return time for $J$. 
We set the {\it return depth} as $r_\omega=\min\{|r|: J_{\omega, k} \cap I_r \neq \emptyset\}$. Notice that $r_\omega$ depends only on $\omega_0, \dots, \omega_{k-1}$. \\

\noindent\textbf{Chopping times.} Suppose that $|J_{\omega, k }|<\delta$, $\Delta_0\cap J_{\omega, k}\neq \emptyset$ and $J_{\omega, k}$  intersects more than three adjacent intervals $I_{r, m}$. Then $k$ is called an {\it essential} return time for $J$.  In this case we chop $J$ into pieces $J_{\omega, r, j}\subset J$ in such that
\[
I_{r, j}\subset T^k_\omega(J_{\omega, r, j})\subset \hat I_{r, j}, 
\]
where $\hat I_{r, j}$ is the union of $I_{r, j}$ and its two neighbours. In this case, we say that $J_{ \omega, r,j}$ has the associated return depth $r_\omega=\min\{|r|: T^k_\omega(J_{\omega, r, j}) \cap I_r \neq \emptyset\}$, and we refer to $J$ as the \textit{ancestor} of $J_{\omega, r, j}$.  \\

\noindent\textbf{Escape times.} Suppose that $k$ is the first free iterate for $J$ such that $|J_{\omega, k}|\ge \delta$. We then set $E_\omega(J)=k$ and add $J$ to $\P^\omega(J_0)$, where we call $J_{\omega, k}$ an \textit{escape interval}. Notice that the partition will depend on $\omega$.  \\

The proof of Proposition 3.1 closely follows the approach of \cite{DO}. Suppose that $J\in \P^\omega(J_0)$, and suppose that $J$ had $s$ returns before escaping. Let us denote $r_i$ as the return depth of the $i$-th return of $J$. Note that $r_i$ will depend on $\omega$, but for ease of notation this is not denoted explicitly. We define the \textit{total return depth} as $R_\omega(J)=\sum_{i=1}^s r_i$.
\begin{lemma}\label{lem2}
There exists $\hat \lambda < \frac{1}{2} $ depending on $\delta$ such that for any $J\in \P^\omega(J_0)$ 
\[
|J|\le  e^{-(1-\hat\lambda)R_\omega(J)}.
\] 
\end{lemma}
\begin{proof}
Suppose that $R_\omega(J)\neq 0$ and that $J$ has $s$ returns before escaping. Let $d_0$ be the number of iterates before the first return, let $d_1, ..., d_{s-1}$ be the free iterates between successive returns or before escaping for $d_s$. Then $E_\omega(J)=d_0+1+ d_1 + 1 + \dots + d_s +1$. We have 
\[
|T_\omega^{E_\omega}(J)|=  |T_\omega^{d_0+1+ d_1 + 1 +\dots + d_s +1}(J)|=
|DT_\omega^{d_0+1+d_1 + 1 + \dots + d_s +1}(\xi)|\cdot|J|, 
\]
for some $\xi\in J$. Notice that by using the chain rule we obtain the following:
\begin{align*}
	DT_\omega^{(d_0+ 1) + (d_1 + 1) + \dots + (d_s+1)}(\xi) &= D(T_{\sigma^{\nu_{s}}\omega}^{d_s + 1} \circ T_\omega^{\nu_{s}})(\xi)\\ 
	&= DT_{\sigma^{\nu_{s}}\omega}^{d_s + 1}(T_\omega^{\nu_{s}}(\xi)) \cdot DT_\omega^{\nu_{s}}(\xi)\\
	&= DT_{\sigma^{\nu_{s}}\omega}^{d_s + 1}(T_\omega^{\nu_{s}}(\xi)) \cdot D(T_{\sigma^{\nu_{s-1}}\omega}^{d_{s-1} + 1} \circ T_\omega^{\nu_{s-1}})(\xi)\\
	&=DT_{\sigma^{\nu_{s}}\omega}^{d_s + 1}(T_\omega^{\nu_{s}}(\xi)) \cdot DT_{\sigma^{\nu_{s-1}}\omega}^{d_{s-1} + 1} (T_\omega^{\nu_{s-1}}(\xi)) \cdot DT_\omega^{\nu_{s-1}}(\xi)\\
	& \vdots \\
	&=  \prod_{i=0}^{s} DT_{\sigma^{\nu_{i}}\omega}^{d_i + 1}(T_\omega^{\nu_{i}}(\xi))\\
	& =DT_\omega^{\nu_1}(\xi) \prod_{i=1}^{s}  DT_{\sigma^{\nu_i + 1} \omega}^{d_i}(T_{\omega}^{\nu_i + 1}(\xi)) \cdot  DT_{\sigma^{\nu_i} \omega}(T_{\omega}^{\nu_{i}}(\xi) ) 
\end{align*}
where $\nu_i = d_0 + 1 + \dots + d_{i-1} + 1$ is the $i$-th return time for $i = 1, \dots, s$ and $\nu_0 = 0$. Notice that for every $i=0, \dots, s$, every $x\in I$ and every $\omega\in\Omega$, we have $DT^{d_i}_\omega(x)\ge \tilde Ce^{\ell d_i}$ by the expansion property from (\ref{eq:unifexp}). Likewise, since $T_\omega^{\nu_i}(x)$ returns for all $x \in J$, and since we can write $T_\omega^{\nu_i}(x) = e^{-r_x}$ for some $r_x >  r_i $, then $DT_{\sigma^{\nu_i}(\omega)}(T_\omega^{\nu_i}(x))\ge C^{-1} e^{(1-\bar\lambda)r_x} \ge C^{-1} e^{(1-\bar\lambda)r_{i}}$ for all $x \in J$ by the order of singularity condition (\ref{eq1}).
Using these two inequalities, we have 
\[
|DT_\omega^{d_0+1+\dots + d_s+1}(\xi)| \ge \tilde C^{s+1} e^{\ell \sum_{i=0}^sd_i}C^{-s}e^{(1-\bar\lambda)\sum_{i=1}^s r_{i}}
>\hat C^{s} e^{(1-\bar\lambda)R_{\omega}}=e^{(1-\bar\lambda +\frac {s}{R_\omega}\log \hat C)R_{\omega}}.
\]
where $\hat C =\min\{C^{-1}, \tilde C^{(s+1)/s}\}$.  

Recall that $r_{ i} \ge r_0$, where $e^{-r_0}=\delta$. Thus $R_\omega\ge sr_0$, which implies $1-\bar\lambda +\frac{s}{R_\omega}\log \hat C\le 1-\bar\lambda +\frac {\log \hat C}{\log \delta^{-1}}$.  Assume we have chosen $\delta$ small enough such that $\hat \lambda=\bar\lambda-\frac {\log \hat C}{\log\delta^{-1}}<\frac{1}{2}$. 
Thus, we have 
\[
1\ge |T_\omega^{E_\omega}(J)|> e^{(1-\hat\lambda)R_{\omega}}|J|
\]
which finishes the proof. 
\end{proof}
The above lemma allows us to prove the following exponential estimate.

\begin{lemma}\label{lem:32}
 Let $J_0\subset I$ and $R_\omega:J_0\to \NN$ be the sum of return depths for all $\omega\in \Omega$. Then we have 
\begin{equation}
|\{R_\omega(J)=k\mid J\in \P^\omega(J_0)\}|\le 10\delta^{-1}e^{-(1-2\hat\lambda)k}|J_0|.
\end{equation}
\end{lemma} 

\begin{proof} Let $\mathcal N_{k}$ denote the set of all sequences of return depths $(r_1, \dots, r_s)$ with $s \ge 1$ such that $r_1+\dots + r_s=k$. Then from \cite[Lemma 3.4]{BLvS}, we know that for sufficiently large $k$ we have
\[
\#\mathcal N_{k}\le e^{\hat\lambda k}.
\]
Furthermore, we know that for any given sequence of return times $(r_1, \dots, r_s)$, there can be at most two escaping intervals. Thus, combining these and Lemma \ref{lem2} we have
\[
\sum_{\substack{J\in \P^\omega(J_0)\\R_\omega(J)=k}}|J| \le
e^{-(1-\hat\lambda)k} \cdot 2 \#\mathcal N_{k}  \le  2e^{-(1-2\hat\lambda)k}\le 10\delta^{-1}e^{-(1-2\hat \lambda)k}|J_0|,
\]
where in the last inequality we have used $|J_0|\ge \delta/5$.
\end{proof}
The following lemma relates the total return depth and the escape time. 
\begin{lemma}
Let $J\in\P^\omega(J_0)$ and let $\mathcal R = (r_{ 1}, \dots, r_{ s})$ be its associated sequence of return depths. Furthermore, let $R_\omega(J)=\sum_{i=1}^s r_i $. Then we have
\[
E_\omega(J)\le \frac{2+ \ell}{\ell}R_\omega(J) + 1.
\]
\end{lemma}
\begin{proof}
Since intervals are not chopped at the inessential return times, we  distinguish between essential and inessential return times. Thus, let $R_\omega(J)=R^e_\omega(J)+R^{ie}_\omega(J)$ be the corresponding splitting into essential and inessential total return depths respectively. Furthermore, let $\mathcal{K}^e= \{\nu_1^e < ... <\nu_q^e \}$ be the ordered set of essential return times, and define $\D^e(J)=d_0^e+\dots + d_q^e$, where $d_0^e$ is the number of free iterates before the first essential return, each $d_i^e $ is the number of free iterates between consecutive essential return times, and $d_q^e$ is the number of free iterates after the last essential return and before the escape time. Likewise, let $\mathcal{K}= \{\nu_1 < ... <\nu_s \}$ and $\D(J)=d_0+\dots +d_s$, where we have already defined the $\nu_i$'s and $d_i$'s in the proof of lemma 3.2. Notice that if $d_{i_j}$ is the number of free iterates between $\nu_{j+1}^e$ and the previous return before $\nu_{j+1}^e$, essential or inessential, for some $i_j = 1, \dots, s-1$, then $d_j^e \ge d_{i_j} $ for all $j=0, \dots, q-1$. But also notice that $\D^e(J)=\D(J)$, thus $q\le s$. 

Recall in the definition of chopping times that we say an interval $\tilde  J \supset J$ is the ancestor of $J$ if $J$ is obtained via chopping of $\tilde J$ at an essential return time. In fact, we can write of sequence of subsets of the form $J = J^{(q)} \subset J^{(q-1)} \subset \dots \subset J^{(1)} \subset J^{(0)}$ such that $J^{(0)}$ is our starting interval and $J^{(i)}$ is obtained by chopping $J^{(i-1)}$ at the $i$-th essential return time. We call $J^{(i)}$ the ancestor of $J$ of order $i$.

Now, let $J^{(i-1)}$ be the ancestor of $J$  of order $i-1$ which is obtained at the essential return time $\nu_{i-1}^e$. We claim that
\begin{align*}
|T_\omega^{\nu_{i}^e}(J^{(i-1)})|\ge 3\frac{e^{-r_i^e}-e^{-(r_i^e + 1)}}{{r_i^e}^{\vartheta}}\ge \frac {e^{-r_i^e}}{{r_i^e}^{\vartheta}},
\end{align*}
where $r_i^e$ is the return depth of the $i$-th essential return. Indeed, the first inequality is true by the definition of essential return. For the second inequality, one can easily show this by calculating the ratio $|I_{r+1}|/|I_r|$, which is just a constant not dependent on $r$, and then taking the geometric sum.

Let $\rho_0$ be the number of inessential returns before the first essential return, let $\rho_i$ denote the number of inessential returns between $\nu_i^e$ and $\nu_{i+1}^e$ for $i=1, \dots, q- 1$, and let $\rho_q$ denote the number of inessential returns after the last essential return and before the escape time.  Now, using the above inequality we have
\begin{align*}
1\ge |T_\omega^{\nu_{i+1}^e}(J^{(i-1)})| 
&= |T_{\sigma^{\nu_{i}^e}\omega}^{d_i^e + \rho_i + 1}(T_\omega^{\nu_{i}^e}(J^{(i-1)}))| \\
&= |T_\omega^{\nu_{i}^e}(J^{(i-1)})|\cdot |DT_{\sigma^{\nu_{i}^e}\omega }^{d_i^e + \rho_i +  1}(T_{\omega}^{\nu_i^e}(\xi))| \\
&= |T_\omega^{\nu_{i}^e}(J^{(i-1)})|\cdot |DT_{\sigma^{\nu_{i}^e + 1}\omega }^{d_i^e + \rho_i }(T_{\omega}^{\nu_i^e + 1}(\xi))| \cdot |DT_{\sigma^{\nu_{i}^e}\omega }(T_{\omega}^{\nu_i^e}(\xi))|  \\
&\ge \frac {e^{-r_i^e}}{{r_i^e}^{\vartheta}} \cdot \tilde C e^{\ell (d^e_i + \rho_i)} \cdot C^{-1} e^{(1-\bar\lambda)r_i^e}\\
& \ge \tilde C  C^{-1}\frac{e^{(1-\bar\lambda)r_i^e}}{{r_i^e}^{\vartheta }}  e^{\ell d^e_i - r_i^e},
\end{align*}
for $i=1, \dots, q$, where we just set $\nu_{q+1}^e= E_\omega(J)$, and where we use $e^{\ell \rho_i} \ge 1$. Now, if $r_0$ is sufficiently large, then $\frac{e^{(1-\bar\lambda)r_i^e}}{{r_i^e}^{\vartheta - 1}} \ge 1$. If $\tilde C  C^{-1} < 1$, then we can simply choose an even larger $r_0$ to cancel these terms out. Otherwise, if $\tilde C  C^{-1} \ge 1$, then we can simply remove these terms. Thus, we obtain $1 \ge e^{\ell d^e_i - r_i^e},$ and taking the log of both sides we obtain $0 \ge \ell d_i^e - r_i^e$, and thus $d_i^e \le \frac{r_i^e}{\ell}$ for $i=1, \dots, q$. Similarly for $d_0^e$, we know that 
\begin{align*}
1 \ge |T_\omega^{\nu_1^e+1}(J^{(0)})| &= |T_\omega^{\nu_1^e}(J^{(0)})| \cdot |DT_{\sigma^{\nu_1^e} \omega}(T_\omega^{\nu_1^e}(\xi))|\cdot |DT_\omega^{\nu_1^e}(\xi)|  \\
	&  \ge C^{-1} \tilde C \frac {e^{-r_1^e}}{{r_1^e}^{\vartheta}}  e^{(1 - \bar \lambda)r_1^e}  e^{\ell(d_0^e + \rho_0 + 1)}.
\end{align*}
Using a similar argument to the previous one above, we obtain $d_0^e \le \frac{r_1^e}{\ell}$.

Note by assumption that there must be $(s-q)$ inessential returns and that there must be $(q+1)$ iterates when a return or the escape happens. Thus, using the fact that $R_\omega(J) \ge s$ and $R_\omega^e (J) \le R_\omega(J)$ we obtain 
\begin{align*}
	E_\omega(J)&= d_0^e + \sum_{i=1}^q d_i^e+(s-q) + (q + 1)\\
	&\le \frac{r_1^e}{\ell} + \sum_{i = 1}^q \frac{r_i^e}{\ell} + R_\omega(J) + 1\\
	&= \frac{r_1^e}{\ell}  + \frac{R_\omega^e(J)}{\ell} + R_\omega(J) + 1\\
	&\le \frac{2+ \ell}{\ell}R_\omega(J) + 1.
\end{align*}
\end{proof} 
Finally, we are ready to prove Proposition \ref{prop:Escape}. 

\begin{proof}[Proof of Proposition \ref{prop:Escape}]
Let us fix some $n \in \mathbb{N}$. By the above lemma, for every $J$ in $\P^\omega(J_0)$ such that $n\le E_\omega(J)$, we therefore have that $n\le  \frac{2+ \ell}{\ell}R_\omega(J) + 1$.  Thus, $R_\omega\ge \frac{(n-1)\ell}{2+\ell}$. Combining this with lemma \ref{lem:32} we have that 
\[
\sum_{\substack{J\in\P^\omega(J_0)\\E_\omega(J)\ge n}}|J|\le\sum_{\substack{J\in\P^\omega(J_0)\\R_\omega\ge \frac{(n-1)\ell}{2+\ell}}}|J|\le
\sum_{j=\frac{(n-1)\ell}{2+\ell}}^\infty 10\delta^{-1}e^{-(1-2\hat\lambda)j}|J_0|\le C_\delta e^{-(1-2\hat\lambda)\frac{(n-1)\ell}{2+\ell}}|J_0|=C_\delta e^{-\gamma n}|J_0|.
\]
Notice that we need to fix $\delta>0$, since $C_\delta=\mathcal O(\delta^{-1}).$ 

Thus, $\P^\omega(J_0)$ defines a partition of $J_0$ and every element of the partition is assigned an escape time.  Notice that the way we constructed the escape time immediately implies that for every $J\in \P^\omega(J_0)$ and time $j\le E_\omega(J)$ such that $T_\omega^j(J)\cap\Delta_0\neq \emptyset$, the image $T_\omega^j(J)$ does not intersect more than three adjacent intervals of the form $I_{r, m}$, $|r|\ge r_0$, $m=1,\dots,  r^\vartheta$.
\end{proof}

\subsection{Bounded Distortion}In this subsection we prove that $T^n_\omega$ has bounded distortion on every interval $J$ such that $T_\omega^k(J)$ does not intersect more than three adjacent intervals $I_{r, m}$ for all $1\le k\le n$. Therefore, the escape map in Proposition \ref{prop:Escape} has bounded distortion. 


We prove the following lemma:
\begin{lemma}\label{lem:bdsm}
For some fixed $\omega \in \Omega$, consider an interval $J \subset I$ such that $|J| < \delta$ and for which there exists $n_J \in \mathbb{N}$ such that for every $0 \leq k \leq n_J - 1$ either $J_{\omega,k} \cap \Delta_0 = \emptyset$ or $J_{\omega, k}$ is contained in at most three intervals $I_{r,m}$. There exists a constant $\mathcal{D} = \mathcal{D}(\delta)$ (independent of $\omega$) such that for every $\omega$ and for every such $J$ described above, we have
	\begin{align}
		\max_{x,y \in J} \log \frac{|DT_\omega^k(x)|}{|DT_\omega^k(y)|} \leq \mathcal{D}, \quad 1 \leq k \leq n_J.
	\end{align}
\begin{proof}
Let $\omega\in\Omega$. If we set $x_{\omega, i} = T_{\omega}^i(x)$, then by using the chain rule  
	and  $\log(1+x)\le x$ for $x\ge 0$ we obtain 
	\begin{align}
		\log \frac{|DT_\omega^k(x)|}{|DT_\omega^k(y)|} &= \log \prod_{i = 0}^{k-1}\frac{ |DT_{\sigma^i \omega}(x_{\omega, i})|}{ |DT_{\sigma^i \omega}(y_{\omega, i})|} 
		\nonumber \\
		& \quad \leq \sum_{i=0}^{k-1} \frac{ |DT_{\sigma^i \omega}(x_{\sigma^i\omega, i})-DT_{\sigma^i \omega}(y_{\omega, i})|}{ |DT_{\sigma^i\omega}(y_{\omega, i})|}. \label{eq4}
	\end{align}
	One can show that
	\begin{align*}
		|DT_{\sigma^i \omega}(x_{\omega, i})-DT_{\sigma^i \omega}(y_{\omega, i})| \le K \frac{|x_{\omega, i}-y_{\omega, i}|^{\alpha}}  {|x_{\omega, i}|^{1 - \lambda_{\sigma^i \omega} }|y_{\omega, i}|^{1 - \lambda_{\sigma^i \omega} } },
	\end{align*}
	where $\lambda_{\sigma^i \omega} \in [\lambda_0, \bar \lambda]$ is the order of singularity associated with the map $T_{\sigma^i \omega}$. Furthermore, we know $|DT_{\sigma^i \omega}(y_{\omega, i})| \geq C^{-1} |y_{\omega, i}|^{\lambda_{\sigma^i \omega} - 1}$, and we also know that $\frac{1}{|x|^{1 - \lambda}} \le \frac{1}{|x|^\alpha}$ for any $x \in I$ because $\alpha \ge 1 - \lambda$ for any $\lambda \in [\lambda_0, \bar \lambda]$. Thus, combining these and setting $\kappa =  K C $, we obtain
	\begin{align}
		\log \frac{|DT_\omega^k(x)|}{|DT_\omega^k(y)|} &\le \sum_{i=0}^{k-1} K \frac{|x_{\omega, i}-y_{\omega, i}|^{\alpha }}  {|x_{\omega, i}|^{1 - \lambda_{\sigma^i \omega}  }|y_{\omega, i}|^{1 - \lambda_{\sigma^i \omega}  } } \cdot \frac{1}{C^{-1} |y_{\omega, i}|^{-(1 - \lambda_{\sigma^i \omega} )} }\nonumber \le  \kappa \sum_{i=0}^{k-1}  \frac{|x_{\omega, i}-y_{\omega, i}|^{\alpha }}  {|x_{\omega, i}|^{\alpha } }.
	\end{align}

	Let $J =[x,y] \subset I$, and let $\mathcal{K}= \{\nu_1 < \nu_2 < ... < \nu_i < ... <\nu_p \}$ be the ordered set of all returns under the dynamics of $T_\omega^k$ on $J$ from $k=0$ to $k=n_J -1$. Note that since we have already assumed there are no essential returns at or before before $n_J-1$, all of these returns must be inessential. The largest possible size for $J_{ \omega, \nu_i}=T_\omega^{ \nu_i }(J)$ is if it is contained in one interval of the form $I_{r_i, m}$ and two of the form $I_{(r_i -1), m}$, thus
	\begin{align}
		|J_{ \omega, \nu_i}| &\leq 2 \frac{|e^{-(r_i - 1)} - e^{-r_i } | }{(r_i -1)^\vartheta} + \frac{|e^{-r_i} - e^{-(r_i + 1)} | }{r_i^\vartheta} \nonumber \\
		&\leq \frac{1}{(r_i -1)^\vartheta} \big(2|e^{-r_i}(e - 1)| +|e^{-r_i}(1 - e^{-1})| \big) \nonumber \\
		&= c_{r_i} \frac{e^{-r_i}}{r_i^\vartheta} \label{eq5}
	\end{align}
	with $c_{r_i}=  \frac{r_i^\vartheta}{(r_i -1)^\vartheta} \big(2e - 1 - e^{-1} \big)$.

	Let us examine time $k$. There are two possibilities: $k \leq \nu_p + 1$ or $\nu_p + 1< k \leq n_J - 1$. We will start with the case that $k \leq \nu_p + 1$. Let $\nu_{s(k)}$ denote the last return time before $k$, and let us split the sum in (\ref{eq4}) accordingly:
	\begin{align}\label{eq:diff}
		\sum_{i=0}^{k-1} \frac{|x_{\omega, i}-y_{\omega, i}|^{\alpha}}  {|x_{\omega, i}|^{\alpha} } = \sum_{i=0, i \notin \mathcal{K}}^{\nu_{s(k)}-1}  \frac{|x_{\omega, i}-y_{\omega, i}|^{\alpha}}  {|x_{\omega, i}|^{\alpha} } &+ \sum_{i=\nu_{s(k)}+1}^{k-1}  \frac{|x_{\omega, i}-y_{\omega, i}|^{\alpha}}  {|x_{\omega, i}|^{\alpha} }  \nonumber \\ &+  \quad \sum_{i=1}^{s(k)}  \frac{|x_{\omega, \nu_i}-y_{\omega,  \nu_i}|^{\alpha}}  {|x_{\omega,  \nu_i}|^{\alpha} }
	\end{align}
	where the first sum is all times up to time $\nu_{s(k)}$ but excluding return times, the second sum is the times after $\nu_{s(k)}$, and the third sum is the return times.
	
	For the first sum, we use the fact that our expansion condition implies that $|J_{\omega, \nu_{s(k)-i}}| \leq \tilde{C}^{-1} e^{-i\ell }|J_{\omega, \nu_{s(k)}}|$ and that $|x_{\omega, \nu_{s(k)} - i}|\geq \delta$ for $i = 1, ..., \nu_{s(k)}$, $(\nu_{s(k)} - i) \notin \mathcal{K}$  to show that
	\begin{align}
		\frac{|x_{\omega,{\nu_{s(k)-i}}}-y_{\omega, {\nu_{s(k)-i}}}|}  {|x_{\omega, {\nu_{s(k)-i}}}|} 
		\leq \tilde{C}^{-1} e^{-\ell i}\frac{|J_{\omega, \nu_{s(k)}}|}{\delta} 
		\le \tilde{C}^{-1} e^{-\ell i} c_{r_{s(k)}} \frac{e^{-r_{s(k)}}}{r_{s(k)}^\vartheta \delta}. \label{eq6}
	\end{align}
	Thus, we have
	\begin{align*}
		 \sum_{i=0, i \notin \mathcal{K}}^{\nu_{s(k)}-1}  \frac{|x_{\omega, i}-y_{\omega, i}|^{\alpha}}  {|x_{\omega, i}|^{\alpha} } 
		& \leq \sum_{\substack{i=1\\(\nu_{s(k)} - i) \notin \mathcal{K} }}^{\nu_{s(k)}}  \tilde{C}^{-\alpha} e^{-\ell (\nu_{s(k)}-i)\alpha}c_{r_{s(k)}}^{\alpha}\frac{e^{-r_{s(k)}\alpha}}{r_{s(k)}^{\vartheta {\alpha}} \delta^{\alpha}} && \Big(\text{using (\ref{eq6})} \Big)  \\
		& \leq C_0 C_1  \frac{c_{r_{s(k)}}^{\alpha}}{r_{s(k)}^{\vartheta {\alpha}} } \sum_{\substack{i=1\\(\nu_{s(k)} - i) \notin \mathcal{K} }}^{\nu_{s(k)}}  e^{-\ell (\nu_{s(k)}-i)\alpha }   \\
		& \leq C_0 C_1   \sum_{i=0}^{\infty} e^{- \ell i \alpha} && \Big(\text{using } \frac{c_{r_{s(k)}}}{r_{s(k)}^{\vartheta } } <1 \Big)    \\
		&=  C_0 \cdot \frac{1}{1 - e^{- \ell \alpha}} = D_1. && \Big(\text{using } \frac{e^{-r_{s(k)}}}{\delta} <1 \Big)   
	\end{align*}
	with $C_0=   \tilde{C}^{-\alpha} $ and $C_1= \Big( \frac{e^{-r_{s(k)}}}{\delta}  \Big)^\alpha$. We note that $D_1$ is independent of both $k$ and $\delta$.

For the second sum, we also have that $|x_{\omega, i}|\geq \delta$ and $|J_{\omega, i}| \leq \tilde{C}^{-1} e^{-\ell (k - 1 - i)}|J_{\omega, k-1}|$ for $i = \nu_{s(k)} + 1, .., k-1.$ Thus
\begin{align*}
\sum_{i=\nu_{s(k)}+1}^{k-1}  \frac{|x_{\omega, i}-y_{\omega, i}|^{\alpha}}  {|x_{\omega, i}|^{\alpha} } &\leq  \sum_{i=\nu_{s(k)}+1}^{k-1}  \tilde{C}^{-\alpha}e^{-\ell(k -1 -i)\alpha}\frac{|J_{\omega,k-1}|^{\alpha}} {\delta^{\alpha}}\\
&  \leq C_0 \sum_{i=\nu_{s(k)}+1}^{k-1}   e^{-\ell(k -1 -i)\alpha}\\ 
&\leq C_0 \sum_{i=0}^{\infty}  e^{- \ell i \alpha}\\
& = C_0 \cdot \frac{1}{1 - e^{-\ell \alpha}} = D_2 < \infty.
	\end{align*}
	
	Finally, for the third sum, we define the set  $\mathcal{K}_r= \{\nu_{\eta_1} < \nu_{\eta_2} < ... < \nu_{\eta_i} < ... <\nu_{\eta_q} \}$ such that for the associated return depths we have $r_{\eta_i}=r$ for all $\nu_{\eta_i}  \in  \mathcal{K}_r$. Clearly
	\begin{align}
		\sum_{i=1}^{s(k)} \frac{ |x_{\omega, \nu_i} - y_{\omega, \nu_i}|^{\alpha}}{|x_{\omega, \nu_i}|^{\alpha}} = \sum_{r \geq r_0} \sum_{i \in \mathcal{K}_r}\frac{ |x_{\omega ,i} - y_{\omega ,i}|^{\alpha}}{|x_{\omega ,i}|^{\alpha}}. \label{eq7}
	\end{align}
	Let $M(r)$ denote the maximum value in $\mathcal{K}_r$. Then we have $|J_{\omega,i}| \leq \tilde{C}^{-1} e^{-\ell (M(r) - i)} |J_{\omega,M(r)}|$ and $|J_{\omega,M(r)}|< c_{M(r)}\frac{e^{-r}}{r^\vartheta}$. Additionally, since $i \in \mathcal{K}_r$ implies $e^{-r}\leq |x_{\omega ,i}| $, we have
	\begin{align}
		\frac{|J_{\omega,i}|}{|x_{\omega,i}|} \leq c_{M(r)} \tilde{C}^{-1} \frac{e^{-\ell(M(r)-i)}}{r^\vartheta}. \label{eq8}
	\end{align}
	Thus, we have 
	\begin{align*}
		\sum_{i \in \mathcal{K}_r} \frac{ |x_{\omega ,i} - y_{\omega ,i}|^{\alpha}}{|x_{\omega ,i}|^{\alpha}} &\leq \sum_{i \in \mathcal{K}_r} c_{M(r)}^{\alpha} \tilde{C}^{-\alpha} \frac{e^{-\ell(M(r)-i)\alpha}}{r^{\vartheta \alpha}} && \Big( \text{using (\ref{eq8})} \Big) \\
		& \leq \frac{C_0 c_{M(r)}^{\alpha}}{r^{\vartheta \alpha}}   \sum_{i \in \mathcal{K}_r}  e^{-\ell(M(r)-i)\alpha} \\
		 &\leq \frac{C_0 c_{M(r)}^{\alpha}}{r^{\vartheta  \alpha}}   \sum_{j=0}^\infty  e^{-\ell j \alpha} \\
		& = \frac{C_0 c_{M(r)}^{\alpha}}{r^{\vartheta  \alpha}}   \frac{1}{1 - e^{-\ell \alpha}}
		 \leq \frac{C_2}{r^{\vartheta  \alpha}}  	
	\end{align*}
	with $C_2= C_0 c_{M(r)}^{\alpha}  \frac{1}{1 - e^{-\ell \alpha}}$. Thus,
	\begin{align*}
		\sum_{i=1}^{s(k)} \frac{ |x_{\omega , \nu_i} - y_{\omega , \nu_i}|^{\alpha}}{|x_{\omega , \nu_i}|^{\alpha}} &=\sum_{r \geq r_0} \sum_{i \in \mathcal{K}_r}\frac{ |x_{\omega ,i} - y_{\omega ,i}|^{\alpha}}{|x_{\omega ,i}|^{\alpha}} \\
		& \leq C_2 \sum_{r \geq r_0}\frac{1}{r^{\vartheta  \alpha}}  = D_3 < \infty.
	\end{align*}
	This concludes the part of the proof for $k \leq \nu_p + 1$.
	
	If we instead have $\nu_p + 1 < k$, then the above proof will not work. This is because, after $\nu_p$, the upper bound of $|J_{\omega,i}| < \delta$ no longer applies so long as the $i$-th iterate does not intersect $\Delta_0$. Instead, we use the fact that, since $J_{\omega,i}$ does not intersect $\Delta_0$, we have $|J_{\omega,i}| < 1 - \delta$ to give us the following estimate:
	\begin{align*}
		 \sum_{i=\nu_p+1}^{k-1}  \frac{|x_{\omega, i}-y_{\omega, i}|^{\alpha}}  {|x_{\omega, i}|^{\alpha} } &\leq  \sum_{i=\nu_p+1}^{k-1}  \tilde{C}^{-\alpha}e^{-\ell (k -1 -i)\alpha} \frac{|J_{k-1}|^{\alpha}} {\delta^{\alpha}} \\
		& < C_0 \sum_{i=\nu_{s(k)}+1}^{k-1}   e^{-\ell (k -1 -i)\alpha}\Big(\frac{1 - \delta} {\delta}\Big)^{\alpha}\\
		&\leq C_0 \Big(\frac{1 - \delta} {\delta}\Big)^{\alpha} \sum_{i=\nu_{s(k)}+1}^{k-1}   e^{-\ell (k -1 -i)\alpha} \Big( \text{since }\Big(\frac{1 - \delta} {\delta}\Big) > 1 \Big) \\
		&= C_0 \Big(\frac{1 - \delta} {\delta}\Big)^{\alpha} \cdot \frac{1}{1 - e^{-\ell  \alpha}} 
		= D_2(\delta) < \infty,
	\end{align*}
	where $D_2(\delta)$ is a function dependent on $\delta$. Thus, we set our distortion constant $\mathcal{D}(\delta) = \kappa (D_1 + D_2(\delta) + D_3)$.
\end{proof}
\end{lemma}

To summarize the last two sections, we have now shown that for a chosen $\delta > 0$ we can construct an escape partition and an escape time function $E_\omega$, the tails for which decay exponentially. Furthermore, since $|J|< \delta$ for every $J\in \P^\omega(J_0)$, then the above bounded distortion result holds on each $J\in \P^\omega(J_0)$.

\subsection{Full return partition}\label{ssec:FR}
Below we construct the full return partition of some neighborhood of $0$. For this we use the following lemma:

\begin{lemma}\label{lem:FR} Let $\delta >0$ be our previously chosen constant. Then there exists some sufficiently small $\delta^\ast > 0$ and some $t^\ast\in \NN$ such that for every $\omega \in \Omega$ and for every interval $J$ with $|J|\ge\delta$ and $J \cap \Delta_0 = \emptyset$ there exists $\tilde J \subset J$ with the following properties: 
\begin{itemize}
\item[(i)] There exists $t < t^*$ such that  $T_\omega^{t}: \tilde J \to  \Delta^*=(-\delta^\ast, \delta^\ast)$ is a diffeomorphism;
\item[(ii)] both components of $J \backslash \tilde J$ have size greater than $\delta/5$; 
\item[(iii)]$|\tilde J| \ge  \beta|J|$, where $\beta$ is a uniform constant.
\end{itemize}
\end{lemma}
\begin{proof}
By assumption (A3) of the unperturbed map, we know that $\bigcup_{n\in\NN}T_0^{-n}(0)$ is dense in $I$. Since all the maps in $\A(\eps)$ are close to $T_0$, for any small constant $\bar\delta >0$ there exists an iterate $t_{\bar\delta}\in\NN$ such that  $\bigcup_{n\le t_{\bar\delta}}(T_\omega^{n})^{-1}(0)$ is $\bar\delta$-dense in $I$  and uniformly bounded away from $0$ for all $\omega$. Let us fix some $0<\bar\delta< \delta/5$, and thereby we also fix some $t_{\bar \delta}$. Then the following holds:
\begin{itemize}
	\item[(a)] if we take the subinterval in the middle of $J$ with length $\bar \delta$, then there is a point $x_*\in J$ which is a preimage of $0$, i.e. there exists some $t \le t_{\bar \delta }$ such that $T_\omega^t(x_*) = 0$;  
	\item[(b)] there exists $\delta^\ast$ sufficiently small such that any connected component of $(T^{n}_\omega)^{-1}( \Delta^\ast)$ has length less than $\bar\delta$ for any $n \le t_{\bar \delta }$, where $\Delta^\ast = (- \delta^\ast,\delta^\ast)$,  and is uniformly bounded away from $0$ for all $\omega\in\Omega$ ;  
	\item[(c)] the distance from $x_*$ to the boundary of $J$ is larger than $\delta/5$. 
\end{itemize}
We let $\tilde J=J \cap (T^{t}_\omega)^{-1}( \Delta^\ast)$. Then item (i) follows automatically. Likewise, item (ii) follows from (c). For item (iii), we know that $T_\omega^t (\tilde J) = \Delta^*$ and $T_\omega^t (J) = \Delta_1$, where $\Delta_1$ is just some larger interval containing $\Delta^*$. We also know by using the chain rule that
\begin{align*}
	| \tilde J | = \frac{|\Delta^*|}{DT_\omega^t(\xi_1)}, \quad
	|J| = \frac{|\Delta_1|}{DT_\omega^t(\xi_2)} 
\end{align*}
where $\xi_1 \in \tilde J$, $\xi_2 \in J$. Thus, 
\begin{align*}
	\frac{| \tilde J |}{|J|} = \frac{|\Delta^*|}{|\Delta_1|} \cdot \frac{DT_\omega^t(\xi_2)}{DT_\omega^t(\xi_1)}.
\end{align*}
To show there is lower bound, $\beta$, for this, we need to show there is an upper bound for $DT_\omega^t(\xi_1)$. Recall that $|J|\ge\delta$, that $T_\omega$ is expanding, and that both components of $J \backslash \tilde J$ have size greater than $\delta/5$. For $t=1$, the lower bound is automatic since $J \cap \Delta_0 = \emptyset$. For $1< t \le t_{\bar \delta}$, if $T_\omega^i (J) \cap \Delta_0 \neq \emptyset$ for some $i < t$, then $|\xi_1| \ge \delta/5$, and thus $DT_\omega^t(\xi_1)$ has an upper bound. This completes the proof
\end{proof}
The above lemma implies that, since each escape interval has length equal to or greater than $\delta$, each escape interval therefore contains a subinterval which maps bijectively onto $\Delta^*$ within a uniformly bounded number of iterates. We  use this property repeatedly in order to construct a full return partition 
$\Q^\omega(\Delta^\ast)$ of  $\Delta^*  =(-\delta^*, \delta^*)$,  i.e. for every $\omega\in\Omega$ we construct a countable  partition $\Q^\omega(\Delta^\ast)$ such that for every $J\in\Q^\omega(\Delta^\ast)$ there exists an associated return time $ \tau_\omega (J)\in\NN$ such that $T_\omega^{\tau_\omega(J)}:J\to \Delta^\ast$ is a diffeomorphism.\footnote{Notice that in principle we need to distinguish between different fibers and consider $T_\omega: I\times\{\omega\} \to I\times\{\sigma \omega\}$. Then we obtain the inducing domain at fiber $\omega$ by $\Delta_\omega^\ast= \Delta^\ast$. Then the induced map is $T_\omega^{\tau(\omega, k)}:  J \times \{\omega \}\to \Delta^\ast_{\sigma^{\tau(\omega, k)}\omega}$. This extension does not cause any problem, since all the estimates are uniform in $\omega$.}.

To this end, we start with an exponential partition of $\Delta^\ast$ of the form $\mathcal{P} (\Delta^*) = \{ I_{r,m} \}_{|r| \ge r_{\ast}}$, where each $ I_{r,m}$ is defined as in the beginning of subsection \ref{ssec:Escape}. We construct an escape partition on each $I_{r, m}$, which by extension naturally induces a partition on $\Delta^\ast$, denoted by $\P_{ \omega, 1}(\Delta^\ast)$. Thus, for every $J\in\P_{{\omega, 1}}(\Delta^\ast)$ we have $J \subset I_{r,m}$ for some $r$ and $m$ which has an associated escape time $E_\omega(J)$. Lemma \ref{lem:FR} implies that each $J\in \mathcal{P}_{{\omega,1}} (\Delta^*)$ contains a subinterval $\tilde J$ that maps diffeomorphically to $\Delta^*$ under some number of iterates bounded above by $E_\omega(J) + t^*$. Thus, in order to construct the next partition $\mathcal{P}_{{\omega,2}}$, we divide $J$ into $\tilde J$ and the components of $J \backslash \tilde J$. To each of these we assign the  first escape time $E_{\omega, 1}=E_\omega(J)$, and we assign to the returning subinterval $\tilde J$ the return time $\tau_\omega(\tilde J) = E_{\omega, 1}(\tilde J) + t(\tilde J)$, where $t(\tilde J)$ is the number of iterates such that $T_{\sigma^{E_{\omega, 1}}\omega}^{t(\tilde J)}(T_\omega^{E_{\omega, 1}} (\tilde J)) = \Delta^*$. We place $\tilde J$ in  $\mathcal{P}_{{\omega, 2}}$, and it will remain unchanged in all subsequent $\mathcal{P}_{{\omega, k}}$'s.

Next, we apply the escape partition algorithm to the connected components of $T_\omega^{E_{\omega,1}}(J \backslash \tilde J)$. This is possible since  Lemma \ref{lem:FR} guarantees that these components will be of size greater than $\delta/5$. Let $J_{NR}$ denote one of the non-returning components of $J \backslash \tilde J$, and let $K \subset J_{NR}$ be the preimage of a subinterval that is obtained after applying the escape partition algorithm to $T_\omega^{E_{\omega,1}}(J_{NR})$. We place $K$ in $\mathcal{P}_{{\omega, 2}}$, and again by Lemma \ref{lem:FR} there exists $\tilde K \subset K$ that also maps diffeomorphically to $\Delta^*$ after a bounded number of iterates. We divide $K$ into $\tilde K$ and the components of $K \backslash \tilde K$, and to each of these we assign the second escape time $E_{\omega, 2} = E_{\sigma^{E_{\omega, 1}}\omega} (T_\omega^{E_{\omega, 1}} K) + E_{\omega, 1}$, and to $\tilde K$ we assign the return time $\tau_\omega(\tilde K)= E_{\omega, 2}(\tilde K) + t(\tilde K)$.

We then repeat this process ad infinitum, by 1) taking each non-returning $L \in \mathcal{P}_{{\omega, k}}$ (for some $k$); 2) chopping $L$ into its returning component and non-returning components; 3) assigning to each of these the $k$-th escape time $E_{\omega, k} = E_{\sigma^{E_{\omega, k-1}} \omega}(T_\omega^{E_{\omega, k-1}}(L)) + E_{\omega, k-1}$; 4) assigning to the returning component $\tilde L$ the return time $\tau_\omega(\tilde L)=E_{\omega, k}(\tilde L) + t(\tilde L) $ and placing $\tilde L$ in $\mathcal{P}_{{\omega, k+1}}$; and 5) performing the escape partition on the remaining non-returning components and placing the resulting subintervals in $\mathcal{P}_{{\omega, k+1}}$.

Note that the $i$-th escape time of a  non-returning interval, $i \leq k$, will be the same as the escape time of its $i$-th ancestor. 

Finally, using the above we define the full partition as
\begin{align}
	\Q^\omega(\Delta^*)= \vee_{k=0}^\infty \mathcal{P}_{{\omega, k}}
\end{align}
i.e. the set of all possible intersections of all $\mathcal{P}_{{\omega, k}}$'s.  
Below we will show that $\Q^\omega$ defines a full-measure partition of $\Delta^\ast$ and set $\tau_\omega(J)=E_{\omega, k}(J)+ t(J)$ for every $J\in \Q^\omega$. 

\begin{remark}
One must note in the above algorithm that, when constructing $\mathcal{P}_{{\omega, k}}$, the choice of $\tilde J \subset J$ which maps diffeomorphically onto $\Delta^*$ is not necessarily unique. However, we also want that if $\tau_\omega(J) = k$ for some $J \in \Q^\omega$, then $\tau_\omega$ only depends on the first $k-1$ elements of $\omega$. Furthermore, we want that if $\omega$ and $\omega'$ share their first $k-1$ entries, then $\tau_\omega(J) = k' \le k$ for some $J \in \Q^\omega$ if and only if  $\tau_{\omega '}(J) = k'$. To ensure this, for any $\omega$'s that share the first $k-1$ entries, we require that the same $\tilde J \subset J$ is chosen for all such $\omega$ when defining $\tau_\omega(\tilde J) \le k$.
\end{remark}

The following lemmas are useful for us.

\begin{lemma}
	Let $\tilde J \subset J \in \mathcal{P}_{{\omega, i}}$ be a non-returning subinterval of $J$ which has had its $i$-th escape. Then we have
	\begin{align}\label{eq:15}
		\sum_{\substack{ \tilde J \subset J \\ E_{\omega, i+1}(\tilde J) \ge E_{\omega, i}(J) + n}}| \tilde J | \leq C_3 e^{- \gamma n} |J|
	\end{align}
where $C_3 = e^\mathcal{D} \cdot C_\delta$ uniformly over $\omega$.
	\begin{proof}
		Since $J \in \mathcal{P}_{{\omega, i}}$, we therefore know that $T_\omega^j (J)$ is contained in at most three intervals of the form $I_{(r,m)}$ for $j \leq E_{\omega, k}$, which means we can apply our bounded distortion Lemma \ref{lem:bdsm}. We make use of the following property of bounded distortion: if $T_\omega$ is of bounded distortion with distortion constant $\mathcal{D}$, for any intervals $A \subset I$ and $B \subset I$ we have
	\begin{align}
		\frac{|A|}{|B|} \leq e^\mathcal{D}\frac{|T_\omega^k(A)|}{|T_\omega^k(B)|}, \label{ineq:BD}
	\end{align}
	where $k \leq \min \{n_A, n_B \}$. Here $n_A \in \mathbb{N}$ is the constant such that for every $0 \leq k \leq n_A - 1$, either $T_\omega^k(A)$ does not intersect $\Delta_0$ or $T_\omega^k(A)$ is contained in at most three intervals of the form $I_{r,m}$, and $n_B$ denotes the same condition for $B$.
	Indeed, since $T_\omega$ is a differentiable function, we know from the mean value theorem that there exist $\xi_1\in A, \xi_2 \in B$ such that
	\begin{align*}
		DT_\omega^k(\xi_1) = \frac{|T_\omega^k(A)|}{|A|}, \quad DT_\omega^k(\xi_2) = \frac{|T_\omega^k(B)|}{|B|}.
	\end{align*}
	Thus,
	\begin{align*}
		\frac{|T_\omega^k(B)|}{|B|} \Big/ \frac{|T_\omega^k(A)|}{|A|} = \frac{| DT_\omega^k(\xi_2)|}{|DT_\omega^k(\xi_1)|} \leq e^\mathcal{D},
	\end{align*}
	and rearranging this we obtain inequality (\ref{ineq:BD}). Using this, we can write
	\begin{align*}
		\sum_{\substack{ \tilde J \subset J \\ E_{\omega, i+1}(\tilde J) \ge E_{\omega, i}(J) + n}} | \tilde J | &=  \sum_{\substack{ \tilde J \subset J \\ E_{\omega, i+1}(\tilde J) \ge E_{\omega, i}(J) + n}} \frac{| \tilde J |}{|J|} \cdot |J|  \\
		& \leq   e^\mathcal{D} \sum_{\substack{ \tilde J \subset J \\ E_{\omega, i+1}(\tilde J) \ge E_{\omega, i}(J) + n}} \frac{| T_\omega^{E_{\omega, i}(J)} (\tilde J) |}{|T_\omega^{E_{\omega, i}(J)} (J)|} \cdot |J|  && \Big( \text{using (\ref{ineq:BD})} \Big) \\
		& \leq   e^\mathcal{D} \sum_{\substack{ \tilde J \subset J \\ E_{\omega, i+1}(\tilde J) \ge E_{\omega, i}(J) + n}} | T_\omega^{E_{\omega, i}(J)} (\tilde J) |  && \Big( \text{using }|J| <  | T_\omega^{E_{\omega, i}(J)} ( J) | \Big)\\
		& \leq   e^\mathcal{D} \cdot C_\delta  e^{- \gamma n} |J|    && \Big( \text{using Proposition \ref{prop:Escape}} \Big) \\
		& =  C_3  e^{- \gamma n} |J|.
	\end{align*}
	\end{proof}
\end{lemma}

An important corollary of this is the following:

\begin{corollary}
	We denote by $Q_\omega^{(n)}( E_{\omega,1}, ...,   E_{\omega,i })$ the set of all $J \in Q^\omega$ such that $J$ has escape times $\{  E_{\omega,1} <  E_{\omega,2} <... <  E_{\omega,i} < n \}$ and whose $(i +1)$-th escape time is after $n$. Then we have 
	\begin{align}\label{ineq:FPtails}
		\sum_{J \in Q_\omega^{(n)} ( E_{\omega,1}, ...,   E_{\omega,i })} |J| \leq C_4^i e^{- \gamma n} | \Delta^*|.
	\end{align}
	\begin{proof}
	For each $J \in  Q_\omega^{(n)} ( E_{\omega,1}, ...,   E_{\omega,i })$ there exists a sequence of ancestors $J \subset J_{(i)} \subset J_{(i-1)} \subset ... \subset J_{(2)} \subset J_{(1)} \subset J_{(0)} =  I_{r,m}$. Note that we can write
	\begin{align*}
		E_{\omega, i + 1} (J) &\geq E_{\omega, i } (J_{(i)}) + (E_{\omega, i + 1} (J) - E_{\omega, i } (J_{(i)})) \\
		E_{\omega, i } (J_{(i)}) &\geq E_{\omega, i-1 } (J_{(i-1)}) + (E_{\omega, i } (J_{(i)}) -  E_{\omega, i-1 } (J_{(i-1)})) \\
		&\vdots \\
		E_{\omega,1 } (J_{(1)}) &\geq E_{\omega,0 } (J_{(0)}) + (E_{\omega,1 } (J_{(1)}) - E_{\omega,0} (J_{(0)}) \\
	\end{align*}
where we set $E_{\omega,0 } (J_{(0)}) = 0$. Let us define $M_j = \{ J \subset J_{(j)} : E_{\omega, j + 1} (J) \geq E_{\omega, j } (J_{(j)}) + (E_{\omega, j + 1} (J) - E_{\omega, j } (J_{(j)}))\}$, and then apply (\ref{eq:15}) recursively: we have
	\begin{align*}
		\sum_{\substack{ J \subset J_{(i)} \\  J \in M_i}} |J| \leq C_3 e^{-\gamma(E_{\omega, i + 1} (J) - E_{\omega, i } (J_{(i)}))}|J_{(i)}|
	\end{align*}	
	and thus by recursion we have
	\begin{align*}
		\sum_{J \in Q_\omega^{(n)}( E_{\omega,1}, ...,   E_{\omega,i })} |J| \leq C_3^i e ^{- \gamma E_{\omega, i + 1}} |I_{(r,m)}| < C_3^i e^{-\gamma n}|\Delta^*|.
	\end{align*}
	\end{proof}
\end{corollary}

Furthermore, we make use of the following lemma, the proof of which can be found in \cite{BLvS}, lemma 3.4:
\begin{lemma}
	Let $\eta \in (0,1)$ and $\mathcal R_{k, q} = \{(n_1, n_2, ..., n_q) :  n_i \geq 1 \Hquad \forall i = 1,..., q : \sum_{i=0}^q n_i = k \}$. Then for $q \leq \eta k$ there exists a positive function $\hat \eta (\eta)$ such that $\hat \eta \to 0$ as $\eta \to 0$, and
	\begin{align}
		\mathcal \# R_{k, q}\le e^{\hat \eta k}. \label{ineq:escapetime}
	\end{align}
\end{lemma}
\subsection{The tail of the return times}\label{ssec:tail}
To establish mixing rates, we need to obtain decay rates for the tails of the return times. By construction, the tail of the return times depend on the number of escape times that occurred before returning. We have the following lemma:
\begin{lemma} \label{C5}  Fix $\omega\in\Omega$ and let $(n_1, n_2, \dots, n_i)$ be the sequence of escape times of an interval $J \in \Q^\omega$  before time $n \geq 1$ such that $\sum_{j=1}^i n_j = n.$ Then there exist constants $B, b >0$ that are uniform in $\omega$ such that
	\begin{align}
		|\{ J \in \Q^\omega : \tau(J) > n \}| \leq Be^{-bn}.
	\end{align}
\end{lemma}
\begin{proof}
	Let us define the following:
	\begin{align}
		&Q^{(n)} = \{ J^\omega \in \Q^\omega : \tau(J) > n \} \\
		&Q_i^{(n)} = \{ J \in Q^{(n)} : E_{\omega, i-1}(J) \leq n < E_{\omega, i}(J)\} \\
		&Q_i^{(n)}(n_1, \dots, n_i) = \{J \in Q_i^{(n)}: \sum_{j=1}^k n_j = E_{\omega, k}(J), 1 \leq k \leq i - 1 \}
	\end{align}
	We decompose $Q^{(n)}$ into the following sums:
	\begin{align*}
		|Q^{(n)}| = \sum_{i \leq n} |Q_i^{(n)}| = \sum_{i \leq \zeta n} |Q_i^{(n)}| +  \sum_{\zeta n < i < n} |Q_i^{(n)}|
	\end{align*}
	where $\zeta$ denotes the proportion of total time $n$ that we consider having few escapes ($< \zeta n$) or many escapes ($> \zeta n$), which the two sums above represent respectively. For the many escapes, we have
	\begin{align}
		|Q_i^{(n)}| \leq ( 1- \beta)^i |\Delta^*| = 2 \delta(1- \beta)^i,
	\end{align}
	where we recall $\beta$, defined in lemma \ref{lem:FR}, is the uniform minimum proportion of size between an escape interval and its returning subinterval. Furthermore, for big enough $i$ we have
	\begin{align*}
		\sum_{\zeta n < i < n} |Q_i^{(n)}| &\leq 2 \delta \sum_{\zeta n < i < n} (1- \beta)^i \leq 2 \delta \sum_{\zeta n < i} (1- \beta)^i  \\
		&\le 2 \delta \sum_{j = 0}^\infty (1- \beta)^{\zeta n + j}  
		\le \frac{2  (1- \beta)^{\zeta n}}{\beta} \leq C_4 e^{- \gamma_\beta n}
	\end{align*}
	where $\gamma_\beta = \zeta (\log ( 1 - \beta) ^{-1})$ and $C_4 = 2/\beta$.

	For the intervals that have few escapes, we have
	\begin{align*}
		 \sum_{i < \zeta n} |Q_i^{(n)}|  = & \sum_{i < \zeta n}  \sum_{\substack{ (n_1, ..., n_i) \\ \sum_{j=1}^i n_j =n }}  |Q_i^{(n)} (n_1, ..., n_i)| \\
		&=\sum_{i < \zeta n} \#\mathcal R_{n, i}  \sum_{J \in Q_i^{(n)}(n_1, \dots, n_i) } |J| && \text{using inequality (\ref{ineq:escapetime})} \\
		& \leq \sum_{i < \zeta n} \#\mathcal R_{n, i} C_3^i e^{- \gamma n} |\Delta^*| 
		 \leq e^{\hat \eta n} e^{- \gamma  n} \sum_{i < \zeta n}   C_3^i  = C_5 e^{\hat \eta n} e^{- \gamma  n}
	\end{align*}	
	with $C_5 = \sum_{i < \zeta n}   C_3^i$. This constant can grow exponentially if $C_3\ge1$, so to avoid this we choose a sufficently small $\zeta$. Combining this with the inequality for many escapes, we have
	\begin{align*}
		|\{ J \in \Q^\omega : \tau(J) > n \}| \leq C_4 e^{- \gamma_\beta n} + C_5 e^{(\hat \eta  - \gamma ) n} \leq Be^{-b n}
	\end{align*}
	for positive constants $B$ and $b$, as long as we have chose small enough $\zeta$.
\end{proof}

Thus, for every $\omega \in \Omega$ we have that $\tau_\omega(x)$ is defined and finite for a.e. $x \in \Delta^*$. It should be emphasized that the rates of decay of the return times are exponential, independently of $\omega$. Furthermore, for every $\omega \in \Omega$, if $J \in \mathcal Q^\omega$ has escape times $(n_1, n_2, \dots, n_i)$ before returning, then for the return time we have $\tau_\omega(J) < n_i + t^*$, where $t^*$ is the constant in lemma \ref{lem:FR} which is independent of $\omega$. 

\subsection{Gibbs-Markov} Let  $\Delta^\ast$, $\Q^\omega$ and $\tau_\omega$ be as in subsection \ref{ssec:FR}. By definition, $T_\omega^{\tau_\omega}:\Delta^\ast\to \Delta^\ast$ has fiber-wise  Markov property: every $J\in\Q^\omega$ is mapped diffeomorphically onto $\sigma^{\tau_\omega(J)}\omega \times \Delta^\ast $. Notice that if $\tau_\omega(J)=k$ then $\tau_\omega$  depends only the first $k-1$ components of $\omega$. This implies that if $\tau_\omega(x)=k$ and $\omega_i'=\omega_i$ for $0\le i\le k-1$ then $\tau_\omega(x)=\tau_{\omega'}(x)$, i.e. $\tau_\omega(x)$ is a stopping time. The uniform expansion is immediate in our case. The tail of the return times are obtained in subsection \ref{ssec:tail}. We still need to show bounded distortion and aperiodicity, which we address below. 

Set $F_\omega(x)=T_\omega^{\tau_\omega(x)}(x)$ for $(\omega, x)\in \Omega \times \Delta^\ast$, and define $\mathcal{Q} = \big\{ \{\omega\} \times J  \Hquad |  \Hquad\omega \in \Omega,  J \in \mathcal{Q}^\omega \big\}.$ Note that $F_\omega^2 (x) = F_{\sigma^{\tau_\omega(x)}\omega} \circ F_\omega (x)$, and if we set $\ell = {{\tau_{\sigma^{\tau_\omega(x)}\omega}(F_\omega(x))  + \tau_\omega(x)}}$, then $F_\omega^3 (x) = F_{\sigma^\ell \omega} \circ F_{\sigma^{\tau_\omega(x)}\omega} \circ F_\omega (x)$, etc.. 


As usual, we introduce a separation time for $x,y\in\mathcal Q^\omega$ by setting 
\[
s_\omega(x, y)=\min\{n\mid F^n_\omega(x)\in J, F^n_\omega(y)\in J', J\neq J', \Hquad J, J' \in\Q \}.
\]
\begin{lemma}
There exists constants $\tilde D$ and $\hat \beta\in (0, 1)$ such that for all $\omega \in \Omega$, for all $J\in \Q^\omega$, and for all $x, y\in J$ 
\[
\log\frac{|DF_\omega(x)|}{|DF_\omega(y)|}\le \tilde D \hat \beta^{s_{\sigma^{\tau_\omega}\omega}(F_\omega(x),F_\omega(y))}.
\]
\end{lemma}
\begin{proof} Let us define
$$\Q_{n}^\omega= \{ J_{i_0} \cap (F_\omega)^{-1}(J_{i_1}) \cap \dots \cap (F_\omega^{n-1})^{-1}(J_{i_{n-1}}) | J_{i_0} \in \mathcal Q^\omega, J_{i_1} \in \mathcal Q^{\sigma^{\tau_\omega(J_{i_0})}\omega}, \dots  \}.$$
Then for every $x, y\in J_{n}\in\Q_{n}^\omega$, $F^i_\omega(x)$ and $F^{i}_\omega(y)$ stay in the same element of $\Q$ for $i=1, 2, \dots, n-1$ . Furthermore, $F^n_\omega(J_{n})=\Delta^\ast$. Define diam$(\Q^\omega)=\sup\{|J|: J\in\Q^\omega\}$. We show that there exists $\bar\kappa\in(0, 1)$ such that  for all $\omega\in\Omega$ and   $J_n\in \Q^\omega_n$ holds  
\begin{equation}\label{eq:size}
|J_n|\le \bar\kappa^{n-1}\diam(\Q^\omega).
\end{equation}
We prove this inequality via induction. For $n=1$ notice that $\Q^\omega_{1}=\Q^\omega$ and proceed as follows: fix $\omega\in\Omega$ and let $J_{2}\in\Q^\omega_{2}$ be such that $J_{2}\subset J_{1}\in\Q^\omega_{1}$. Since $F_\omega  J_{1}=\Delta^\ast$ and $F_\omega  J_{2}\in\Q_1^\omega$. we have 
\[
\frac{|J_{1}\setminus J_{2}|}{|J_{1}|}\ge e^{-\D}\frac{|F_\omega(J_{1}\setminus J_{2})|}{|F_\omega(J_{1})|}\ge e^{-\D}\frac{2\delta^\ast-\diam(\Q^\omega)}{2\delta^\ast}=\bar\kappa\in (0, 1).
\]
Thus,  we obtain $|J_{2}| \le (1-\bar\kappa)|J_{1}|$.  Iterating the the process we obtain \eqref{eq:size}.

Consider  $x,y\in J_1$  with $n=s_\omega(x,y)$ and $F_\omega(x), F_\omega(y)\in J_n\in\Q^\omega_{n}$ and suppose that $ J_{n}\subset J\in\Q^{\omega}$. By equation \eqref{ineq:BD} and Lemma \ref{lem:bdsm} we have 
$\frac{|T_\omega^k(J_n)|}{|T_\omega^k(J_1)|}\le e^{\D} \frac{|F_\omega(J_n)|}{|F_\omega(J_1)|}$ for all $k\le n$ and $\omega\in\Omega$. Thus proceeding as in the proof of Lemma \ref{lem:bdsm} we have 
\[\begin{aligned}
\log\frac{|DF_\omega(x)|}{|DF_\omega(y)|}=\sum_{i=0}^{\tau_\omega(J_1)}\frac{|T^i_\omega(J_n)|^\alpha}{|T^i_\omega(y)|^\alpha}\le e^{\alpha\D}\sum_{i=0}^{\tau_\omega(J_1)}\frac{|T^i_\omega(J_1)|^\alpha}{|T^i_\omega(y)|^\alpha} \frac{|F_\omega(J_n)|^\alpha}{|F_\omega(J_1)|^\alpha}\\
\le \frac{e^{\alpha\D}}{2\delta^{\ast\alpha}}|F_\omega(J_n)|^\alpha\sum_{i=0}^{\tau_\omega(J_1)}\frac{|T^i_\omega(J_1)|^\alpha}{|T^i_\omega(y)|^\alpha}.
\end{aligned}\]
The sum on the right hand side can be bounded by a constant $\D$ as in the proof of Lemma \ref{lem:bdsm}.  Therefore, using \eqref{eq:size} and 
$\diam(\Q^\omega)\le 2\delta^\ast$ we have 
\[
\log\frac{|DF_\omega(x)|}{|DF_\omega(y)|}\le e^{\alpha\D}\D \bar\kappa^{\alpha n}=\tilde D \hat \beta^{s_\omega(F_\omega(x),F_\omega(y))}. 
\]
for suitable constant $\tilde D$ and $\hat \beta$.  
\end{proof}
\subsection{Aperiodicity} Finally we address the problem of aperiodicity; i.e.  there exists $N_0\in \NN$ and two sequence $\{t_i\in\NN, i=1,2, \dots, N_0\}$, $\{\eps_i>0, i=1, \dots, N_0\}$ such that g.c.d.$\{t_i\}=1$ and $|\{x\in\Delta^\ast\mid \tau_\omega(x)=t_i\}|>\eps_i$ for almost every $\omega\in\Omega$.  To show this, we recall that the original Lorenz system and all sufficiently close systems are mixing \cite{LMP}. Thus, we can proceed as in \cite[Remark 3.14]{ABR}. Since the unperturbed map admits a unique invariant probability measure, it can be lifted to the induced map over $\Delta^\ast$ constructed following the algorithm in the previous 2 subsections. Moreover, the lifted measure is invariant and mixing for the tower map. Therefore, there exists partition $\Q^0$ of $\Delta^\ast$ and a return time $\tau^0:\Delta^\ast\to \NN$ such that $\tau_i^0=\tau^0(Q_i)$, $Q_i\in\Q^0$ such that g.c.d.$\{\tau_i^0\}_{i=1}^{N_0}=1$ for some $N_0>1$. Now, by shrinking $\eps$ if necessary, we can ensure that the first $N_0$ elements of the partition $\Q^\omega$  satisfy $|Q_i^\omega\cap Q_i|\ge |Q_i|/2$ with $\tau_\omega^0(Q_i^\omega)=\tau_i^0$ for $i=1, \dots N_0$ and for all $\omega\in\Omega$. Thus, we may take $\eps_i=|Q_i|/2$. Notice that we define only finitely many domains in this way. Therefore, the tails estimates, distortion, etc. are not affected, and we can stop at some $\eps_0>0$, thereby proving aperiodicity.\\

Finally, we can now give the proof of Theorem \ref{thm:1d} by applying the results of \cite{BBMD} and \cite{Du}:

\begin{proof} (Proof Theorem \ref{thm:1d})
In the above we have shown that conditions (C2)-(C6) of the Appendix hold. Condition (C1) is true by construction. Using this and Theorem \ref{thm:randexp} in the Appendix we are now ready to prove Theorem \ref{thm:1d}. We begin the proof by defining the \textit{tower projection} $\pi_\omega : \Delta_\omega \to I$ for almost every $\omega \in \Omega$ as $\pi_\omega (x, \ell) = T_{\sigma^{-\ell}\omega}^\ell (x)$. One should note that $\pi_{\sigma \omega} \circ \hat F_\omega = T_\omega \circ \pi_\omega $, where $\hat F_\omega$ is the tower map defined in the appendix. Then $\mu_\omega=(\pi_{\omega})_*\nu_\omega$ provides an equivariant family of measures for $\{T_\omega\}$ and the absolute continuity follows from the fact that $T_\omega$ are non-singular. Now, `lift' the observables $\varphi \in L^\infty(X)$ and $\psi \in C^\eta(X)$ to the tower. Let  $\bar{\varphi}_\omega = \varphi \Hquad \circ \pi_\omega$ and $\bar{\psi}_\omega = \psi \Hquad \circ  \pi_\omega$ respectively. Now, using the definition of $\mu_\omega$ and that we have
\begin{align*}
	 \int (\varphi \circ T^n_\omega) \psi d\mu_\omega &=  \int (\varphi \circ T^n_\omega \circ \pi_\omega) (\psi \circ \pi_\omega) d\nu_\omega \\
 			&= \int (\varphi \circ  \pi_{\sigma^n \omega} \circ \hat F_\omega^n) \bar \psi_\omega d\nu_\omega \\
	&= \int (\bar{\varphi}_{\sigma^n \omega}  \circ \hat F_\omega^n) \bar \psi_\omega h_\omega dm,
\end{align*}
where $d\nu_\omega = h_\omega \cdot dm$. Equally, we have
\begin{align}
	\int \varphi d\mu_{\sigma^n \omega} 
	= \int (\varphi \circ \pi_{\sigma^n \omega}) d\nu_{\sigma^n \omega}  
	= \int \bar \varphi_{\sigma^n \omega}   d\nu_{\sigma^n \omega}
\end{align}
and
\begin{align}
	\int \psi d\mu_\omega 
	= \int (\psi \circ \pi_\omega)  d\nu_\omega 
	= \int \bar \psi_\omega  h_\omega dm.
\end{align}
Thus, we have
\begin{align*}
\Big| \int (\varphi \circ T_\omega^n) \psi d\mu_\omega - \int \varphi d\mu_{\sigma^n \omega} &\int \psi d\mu_\omega \Big| = \\
&\Big| \int (\bar{\varphi}_{\sigma^n \omega} \circ \hat F_\omega^n) \bar{\psi} h_\omega dm - \int \bar \varphi_{\sigma^n \omega} d\nu_{\sigma^n \omega} \int \bar \psi_\omega h_\omega dm \Big|.
\end{align*}
This means that if we can show that $\bar{\varphi}_\omega \in \mathcal{L}_\infty^{K_\omega}$ and $\bar \psi_\omega h_\omega \in \mathcal{F}_\gamma^{K_\omega}$, then we can apply Theorem (\ref{thm:randexp}) in the appendix and thereby obtain exponential decay of correlations on the original dynamics. The first condition is trivial since $\varphi \in L^\infty(X)$. To show the second condition, since $F_\omega$ is uniformly expanding; i.e., $|(T_\omega)'|\ge\kappa>1$, we have $|x-y| \le {(\frac{1}{\kappa})}^{s_\omega(x, y)}$. Hence, for any $(x, \ell), (y, \ell)\in \Dom$ we have the following inequality: 
\begin{equation}\label{eq:step1}
|\bar\psi (x, \ell)- \bar\psi (y, \ell)| \leq \|\psi\|_\eta ({\frac{1}{\kappa}})^{\eta\cdot s_\omega(x,y)}\le\|\psi\|_\eta(\frac{1}{\kappa^\eta\gamma})^{s_\omega(x,y)}\gamma^{s_\omega(x,y)},
\end{equation}
where $\| \cdot \|_\eta$ is the H\"older norm with exponent $\eta \in (0,1)$. 

Let us define $\gamma =\frac{1}{\kappa}$, and let take the separation time on the tower, $\hat s_\omega : \Delta_\omega \times \Delta_\omega \to \mathbb{Z}_+$, which is defined in the appendix. Notice that $\hat s_\omega((x, \ell), (y, \ell)) = s(x,y)$. Thus, inequality \eqref{eq:step1} implies
\begin{align*}
&|(\bar\psi h_\omega)(x, \ell)- (\bar\psi h_\omega)(y, \ell)| \le {\|\psi\|}_{\infty} {\|h_\omega\|}_{\mathcal F_\gamma^{K_\omega}}\gamma^{\hat s_\omega((x, \ell), (y, \ell))}   
\\& + {\|h_\omega\|}_{\mathcal L_\infty}\|\psi\|_\eta\gamma^{\hat s_\omega((x, \ell), (y, \ell))}.
\end{align*}
This completes the proof.
\end{proof}

\section{Appendix}\label{appendix}
\subsection{Abstract random towers with exponential tails}
This appendix is based on the work and results of \cite{BBMD, Du}. See also the related work of \cite{BBR, LV} on random towers. Using the notation and definitions in section \ref{proofs}, we can introduce a random tower for almost every $\omega$ as follows:
\[
	\Delta_\omega = \{ (x, \ell) \in \Delta^{*} \times \mathbb{Z}_+ \Hquad | \Hquad  x  \in \Q^{\sigma^{-\ell} \omega}(\Delta^*), \Hquad j,  \ell \in \mathbb{N}, \Hquad 0 \le \ell \le \tau_{\sigma^{-\ell}\omega}(x) - 1\}.
\]
We can also define the induced map $\hat F_\omega : \Delta_\omega \to \Delta_{\sigma \omega}$ as
\[
	\hat F_\omega(x, \ell) = \begin{cases}
				(x, \ell + 1), \quad &\ell + 1 < \tau_{\sigma^{-\ell}\omega}(x) \\
				(T_{\sigma^{-\ell}\omega}^{\ell + 1}(x), 0), &\ell + 1 = \tau_{\sigma^{-\ell}\omega}(x)				
			\end{cases}.
\]
Notice that this allows us to construct a partition on the random tower as
\[
	\mathcal{Z}_\omega= \{ \hat F_{\sigma^{-\ell}\omega}^\ell (J^{\sigma^{-\ell}\omega}) \Hquad  | \Hquad J^{\sigma^{-\ell}\omega} \in \Q^{\sigma^{-\ell}\omega} (\Delta^*), \Hquad  \tau_\omega |{J^{\sigma^{-\ell}\omega}} \ge \ell + 1, \Hquad \ell \in \mathbb{Z}_+ \}.
\]
Let us define the separation time on the tower as
$$
	\hat s_\omega((x,\ell),  (y, \ell))=\min\{n\mid \hat F^n_\omega(x, \ell )\in J, \hat F^n_\omega(y, \ell)\in J', J\neq J' \in \mathcal{Z}_\omega \}.
$$

Assume:
\begin{itemize}
	\item[(C1)] \textbf{Return and separation time:} the return time function $\tau_\omega$ can be extended to the whole tower as $\tau_\omega: \Delta_\omega \to \mathbb{Z}_+$ with $\tau_\omega$ constant on each $J \in \mathcal{Z}_\omega$, and there exists a positive integer $p_0$ such that $\tau_\omega \ge p_0$. Furthermore, if $(x, \ell)$ and $(y, \ell)$ are both in the same partition element $J \in \mathcal{Z}_\omega$, then $\hat s_\omega((x,0), (y,0)) \ge \ell$, and for every $(x, 0), (y,0) \in J \in \mathcal{Z}_\omega$ we have  
	\[
		\hat s_\omega((x,0), (y,0))= \tau_\omega(x, 0) + \hat s_{\sigma^{\tau_\omega}\omega}(\hat F_\omega^{\tau_\omega (x,0)}(x,0), \hat F_\omega^{\tau_\omega (y,0)}(y,0))
	\]
	\item[(C2)] \textbf{Markov property:} for each $J \in \Q^\omega (\Delta^*)$ the map $\hat F_\omega^{\tau_\omega} |J  : J \to \Delta^{*}$ is bijective and both $\hat F_\omega^{\tau_\omega} |J$ and its inverse are non-singular.
	\item[(C3)] \textbf{Bounded distortion:} There exist constants $0 < \gamma < 1$ and $\mathcal{D}> 0$ such that for all $J \in \mathcal{Z}_\omega$ and all $(x, \ell), (y,\ell) = x, y \in J $
		\[
			\Big| \frac{J\hat F_\omega^{\tau_\omega}(x)}{J\hat F_\omega^{\tau_\omega}(y)} - 1 \Big| \le \mathcal{D} \gamma^{\hat s_{\sigma^{\tau_\omega}\omega}(\hat F_\omega^{\tau_\omega}(x,0), \hat F^{\tau_\omega}_\omega(y,0))}.
		\]
	where $J\hat F_\omega^{\tau_\omega}$ denotes the Jacobian of $\hat F_\omega^{\tau_\omega}$.
	\item[(C4)] \textbf{Weak forwards expansion:} the diameters of the partitions $\bigvee^n_{j=0}  (\hat F_\omega^j)^{-1} \mathcal{Z}_{\sigma^j \omega}$ tend to zero as $n \to \infty$.
	\item[(C5)] \textbf{Return time asymptotics:} there exist constants $B, b > 0$, a full-measure subset $\Omega_1 \subset \Omega$ such that for every $\omega \in \Omega_1$ we have
\[
		m( \{ x \in \Delta^{*} | \tau_\omega > n \}) \le Be^{-bn}
\]
We also have for almost every $\omega \in \Omega$
\[
	m(\Delta_\omega) = \sum_{\ell \in \mathbb{Z}_+} m( \{\tau_{\sigma^{-\ell}\omega} > \ell \}) < \infty
\]
which gives us the existence of a family of finite equivariant sample measures. We also have for almost every $\omega \in \Omega$
\[
	\lim_{\ell_0 \to \infty} \sum_{\ell \ge \ell_0} m(\Delta_{\sigma^{\ell_0} \omega, \ell}) = 0.
\]

	\item[(C6)] \textbf{Aperiodicity:} there exists $N_0 \ge 1$, a full-measure subset $\Omega_2 \subset \Omega$ and a set $\{ t_i \in \mathbb{Z}_+ : \Hquad i = 1,2,...,N \}$ such that g.c.d.$\{ t_i \} = 1$ and there exist $\epsilon_i >0$ such that for every $\omega \in \Omega_2$ and every $i \in \{1,...,N_0\}$ we have $m( \{ x \in \Lambda | \Hquad \tau(x) = t_i \}) > \epsilon_i$. \\
\end{itemize}
Let us also define the following function spaces:
\begin{align}
	\mathcal{F}_\gamma^+ = \{ \varphi_\omega : \Delta_\omega \to \mathbb{R} \Hquad | \Hquad &\exists C_\varphi > 0 \text{ such that } \forall J \in \mathcal{Z}_\omega, \text{ either } \varphi_\omega|J \equiv 0 \\
	&\text{or }\varphi_\omega|J>0 \text{ and } \Big| \log \frac{\varphi_\omega(x)}{\varphi_\omega(y)} \Big| \le C_{\varphi} \gamma^{s_\omega(x,y)}, \forall x,y \in J \}. \nonumber
\end{align}
For almost every $\omega$ let $K_\omega: \Omega \to \mathbb{R}_+$ be a random variable which satisfies $\inf_\Omega K_\omega > 0$ and
\[
	\mathbb{P} (\{ \omega \Hquad | \Hquad K_\omega > n \}) \le e^{-un^v},
\]
where $u,v>0$. We then define the spaces
\begin{align}
	\mathcal{L}^{K_\omega}_\infty &= \{ \varphi_\omega: \Delta_\omega \to \mathbb{R} \Hquad | \Hquad \exists C_{\varphi}^{'} > 0, \sup_{x \in \Delta_\omega}|\varphi_\omega(x)| \le C_{\varphi}^{'} K_\omega \} \\
	\mathcal{F}_\gamma^{K_\omega} &= \{ \varphi_\omega \in \mathcal{L}^{K_\omega}_\infty \Hquad | \Hquad \exists C_{\varphi}'' > 0, |\varphi_\omega(x) - \varphi_\omega(y)| \le C_{\varphi}'' K_\omega \gamma^{s_\omega(x,y)}, \forall x, y \in \Delta_\omega \}
\end{align}
We assign to $\mathcal{L}^{K_\omega}_\infty$ and $\mathcal{F}_\gamma^{K_\omega}$ the norms $|| \varphi ||_{\mathcal{L}_\infty} = \inf C_\varphi '$ and $||\varphi||_\mathcal{F} = \max \{ \inf C_{\varphi_\omega}', \inf C_{\varphi_\omega}''\}$ respectively, which makes them Banach spaces.\\ \\
We can now apply the following theorem from \cite{Du, BBMD}:

\begin{theorem}\label{thm:randexp}
Let $\hat F_\omega$ satisfy (C1)-(C6), and let $K_\omega$ satisfy the above condition. Then for almost every $\omega \in \Omega$ there exists an absolutely continuous $\hat F_\omega$-equivariant probability measure $\nu_\omega = h_\omega m$ on $\Delta_\omega$, satisfying $(\hat F_\omega)_* \nu_\omega = \nu_{\sigma \omega}$, with $h_\omega \in \mathcal{F}_\gamma^+$. Furthermore, there exists a full-measure subset $\Omega_2 \subset \Omega$ such that for every $\omega \in \Omega_2$, $\varphi_\omega \in \mathcal{L}_\infty^{K_\omega}$ and $\psi_\omega \in \mathcal{F}_\gamma^{K_\omega}$ there exists a constant $C_{\varphi, \psi}$ such that for all $n$ we have
\begin{align}\label{ineq:future}
		\Big| \int (\varphi_{\sigma^n \omega} \circ \hat F_\omega^n) \psi_\omega dm - \int \varphi_{\sigma^n \omega} d\nu_{\sigma^n \omega} \int \psi_\omega dm \Big| \le  C_{\varphi, \psi} e^{-bn}
\end{align}
and 
\begin{align}\label{ineq:past}
		\Big| \int (\varphi_\omega \circ \hat F_{\sigma^{-n} \omega}^n) \psi_{\sigma^{-n} \omega} dm - \int \varphi_\omega d\nu_\omega \int \psi_{\sigma^{-n} \omega} dm \Big| \le  C_{\varphi, \psi} e^{-bn}.
\end{align}

\end{theorem}
Usually there would be a constant $C_\omega$ dependent on $\omega$ in front of the right hand side for both \ref{ineq:future} and \ref{ineq:past} because there is often a waiting time dependent on $\omega$ before we see exponential tails of return times. However, since we have uniform tails in (C5), in the sense that we do not have a waiting time and instead we immediately have exponential tails of returns, we just have a uniform constant $B>0$ instead of $C_\omega$, which we can just absorb into $C_{\varphi, \psi}$.

\end{document}